\documentclass{amsart}

\usepackage{amssymb,amsthm,amstext,amsmath,amscd,latexsym,amsfonts,amscd,enumerate}
\usepackage[usenames]{color}
\usepackage{braket}
\usepackage{graphicx}
\newtheorem{theorem}{Theorem}[section]
\newtheorem{proposition}[theorem]{Proposition}
\newtheorem{lemma}[theorem]{Lemma}
\newtheorem{corollary}[theorem]{Corollary}
\theoremstyle{definition}

\newtheorem{remark}[theorem]{Remark}
\newtheorem{example}[theorem]{Example}
\newtheorem{question}[theorem]{Question}
\theoremstyle{remark}
\newtheorem*{ac}{Acknowledgments}
\numberwithin{equation}{theorem}
\newcommand{\Spec}{\mathop{\mathrm{Spec}}\nolimits}
\newcommand{\Supp}{\mathop{\mathrm{Supp}}\nolimits}

\newcommand{\Ass}{\mathop{\mathrm{Ass}}\nolimits}
\newcommand{\Assh}{\mathop{\mathrm{Assh}}\nolimits}
\newcommand{\soc}{\mathop{\mathrm{soc}}\nolimits}
\newcommand{\Min}{\mathop{\mathrm{Min}}\nolimits}
\newcommand{\ann}{\mathop{\mathrm{ann}}\nolimits}

\newcommand{\codim}{\mathop{\mathrm{codim}}\nolimits}

\newcommand{\RHom}{\mathop{\mathrm{\boldsymbol{\rm R}Hom}}\nolimits}
\newcommand{\Lotimes}{\mathop{\otimes^{\boldsymbol{\rm L}}_R}\nolimits}

\newcommand{\depth}{\mathop{\mathrm{depth}}\nolimits}
\newcommand{\ldim}{\mathop{\mathrm{ldim}}\nolimits}

\newcommand{\pd}{\mathop{\mathrm{pd}}\nolimits}

\newcommand{\Td}{\mathop{\mathrm{Td}}\nolimits}
\newcommand{\amp}{\mathop{\mathrm{amp}}\nolimits}
\newcommand{\cmd}{\mathop{\mathrm{cmd}}\nolimits}
\newcommand{\lcmd}{\mathop{\mathrm{lcmd}}\nolimits}
\newcommand{\rank}{\mathop{\mathrm{rank}}\nolimits}

\newcommand{\fp}{\mathfrak{p}}

\newcommand{\fq}{\mathfrak{q}}
\newcommand{\fn}{\mathfrak{n}}
\newcommand{\fm}{\mathfrak{m}}

\def\H{\operatorname{\mathsf{H}}}
\def\K{\operatorname{\mathsf{K}}}
\def\xx{\boldsymbol{x}}
\def\D{\operatorname{\mathsf{D}}}
\def\R{\mathbf{R}}
\def\I{\operatorname{I}}
\def\g{\operatorname{G^\ast-dim}}

\newcommand{\jump}[1]{\ensuremath{[\![#1]\!]} }

\begin{document}
\title{Little dimension and the improved new intersection theorem}
\author{Tsutomu Nakamura}
\address[T. Nakamura]{Dipartimento di Informatica - Settore di Matematica, Universit\`a degli Studi di Verona, Strada le Grazie 15 - Ca' Vignal, I-37134 Verona, Italy}
\email{tsutomu.nakamura@univr.it}
\author{Ryo Takahashi}
\address[R. Takahashi]{Graduate School of Mathematics, Nagoya University, Furocho, Chikusaku, Nagoya, Aichi 464-8602, Japan/Department of Mathematics, University of Kansas, Lawrence, KS 66045-7523, USA}
\email{takahashi@math.nagoya-u.ac.jp}
\urladdr{https://www.math.nagoya-u.ac.jp/~takahashi/}
\author{Siamak Yassemi}
\address[S. Yassemi]{School of Mathematics, Statistics and Computer Science, College of Science, University of Tehran, Tehran, Iran}
\email{yassemi@ut.ac.ir}
\subjclass[2010]{13C15, 13D22}
\keywords{little dimension, improved new intersection theorem, big Cohen--Macaulay module}
\thanks{Ryo Takahashi was partly supported by JSPS Grant-in-Aid for Scientific Research 16K05098 and JSPS Fund for the Promotion of Joint International Research 16KK0099}
\begin{abstract} 
Let $R$ be a commutative noetherian local ring. We define a new invariant for $R$-modules which we call the little dimension.
Using it, we extend the improved new intersection theorem.
\end{abstract}
\maketitle

\section{Introduction}

Throughout this paper, we assume that $R$ is  a commutative noetherian local ring with maximal ideal $\fm$ and residue field $k$. 
An $R$-module $B$ is called a {\em balanced big Cohen--Macaulay module} if $B\ne\fm B$ and any system of parameters of $R$ is a regular sequence on $B$; see \cite[\S 8.5]{BH}.
By a recent work of Andr\'e \cite{A} (see also \cite{HM}), any commutative noetherian local ring admits a balanced big Cohen--Macaulay module, and thus the following {\em improved new intersection theorem} holds (see \cite[p.509]{H} and \cite[p.153]{H2}).

\begin{theorem}[Andr\'e \cite{A}, Evans--Griffith \cite{EG}]\label{INIT}
Let $F=(0\to F_n\to F_{n-1}\to \cdots \to F_{0}\to 0)$ be a complex of free $R$-modules of finite rank.
Assume that  
\begin{itemize}
\item[(1)]
$\H_i(F)$ has finite length for all $i>0$, and
\item[(2)]
there is an element $x\in\H_0(F)\setminus \fm\H_0(F)$ such that $Rx$ has finite length.
\end{itemize}
Then $\dim R\leq n$.
\end{theorem}

In this paper, we extend this theorem by using a new invariant for modules; we define the {\em little dimension} of an $R$-module $M$ as 
$$
\ldim_R M=\inf\{\dim_R Rx\mid x\in M\setminus \fm M\}.
$$
Note that we have $\ldim_R M\leq \dim_R M$ if $M\neq \fm M$, and $\ldim_R M=\infty$ otherwise.

To state our main theorem, we introduce some notation. 
Let $X=(\cdots \to X_{i+1}\to X_{i}\to X_{i-1}\to \cdots)$ be a complex of $R$-modules.
The {\em supremum}, {\em infimum} and {\em amplitude} of $X$ are defined by
$$
\sup X=\sup\{n\mid\H_n(X)\neq 0\},\quad
\inf X=\inf\{n\mid\H_n(X)\neq 0\},\quad
\amp X= \sup X- \inf X.
$$
The {\em support} of an $R$-module $M$, denoted by $\Supp_RM$, is defined as the set of prime ideals $\fp$ of $R$ with $M_\fp\neq 0$.
The following theorem is the main result of this paper.

\begin{theorem}\label{main theorem}
Let $F=(0\to F_n\to F_{n-1}\to \cdots \to F_{0}\to 0)$ be a complex of flat $R$-modules such that $k \otimes_R F$ has nontrivial homology.
Suppose that either
\begin{itemize}
\item[(a)]
$\amp (B\otimes_R F)=0$ for some balanced big Cohen--Macaulay module $B$, or
\item[(b)]
$\Supp\H_i(F)\subseteq \{\fm\}$ for all $i>0$. 
\end{itemize}
Then there is an inequality
$$
\dim R\leq \ldim_R\H_0(F)+\sup (k\otimes_R F).
$$
\end{theorem}

Under the assumption of Theorem \ref{INIT}, one has $\ldim_R\H_0(F)=0$ and $\sup (k\otimes_R F)\leq n$, and the condition (b) in Theorem \ref{main theorem} is satisfied. 
Thus Theorem \ref{main theorem} extends Theorem \ref{INIT} based on the existence of balanced big Cohen--Macaulay modules.
We also emphasize that the main theorem can treat complexes of infinitely generated flat $R$-modules.

In Section 2, we prove our main Theorem \ref{main theorem}.
In Section 3, we discuss the relationship between little dimensions and Cohen--Macaulay modules.
Section 4 contains some examples concerning little dimensions.
We also observe that the inequality of Theorem \ref{main theorem} can be an equality or a strict inequality.
 
\section{Proof of the main theorem}

Recall that the {\em (Krull) dimension} of an $R$-module $M$, denoted by $\dim_RM$, is defined as the supremum of the lengths of chains of prime ideals in $\Supp_RM$.
Also, the {\em depth} of an $R$-complex $X$ is defined by
$$
\depth_RX=n-\sup(\K(\xx)\otimes_RX)=-\sup \RHom_R(k,X)=-\sup\R\Gamma_\fm(X),
$$
where $\xx=x_1,\dots,x_n$ is a system of generators of $\fm$, and $\K(\xx)$ stands for the Koszul complex on $\xx$; we refer the reader to \cite[Theorem I]{FI} for details.
We give a couple of properties of the little dimension.

\begin{proposition}\label{depth-ldim-dim}
For each $R$-module $M$ satisfying $M\neq \fm M$ one has $\depth_R M \leq \ldim_RM \leq \dim_R M$.
\end{proposition}

\begin{proof}
The second inequality is clear, which already appeared in the previous section.
To show the first one, choose an element $x\in M\setminus \fm M$ with $\ldim_RM=\dim_R Rx$.
Then there is an ideal $I$ of $R$ such that $Rx\cong R/I$ and $Ix=0$.
Thus the first inequality follows from \cite[Lemma 3.3]{I}.
\end{proof}

\begin{lemma}\label{ldim-tensor}
Let $M$ and $N$ be $R$-modules such that $M\neq \fm M$ and $N\neq \fm N$.
\begin{enumerate}[\rm(1)]
\item
Let $x\in M\setminus \fm M$ and $y\in  N\setminus \fm N$.
Then $x\otimes y\in(M\otimes_RN) \setminus \fm (M\otimes_RN)$.
In particular, $M\otimes_RN\ne0$.
\item
There is an inequality $\ldim_R(M\otimes_R N)\leq \min\{\ldim_RM,\,\ldim_RN\}$.
\end{enumerate}
\end{lemma}

\begin{proof}
(1) The elements $\overline{x}\in M/\fm M$ and $\overline{y}\in N/\fm N$ are nonzero, and so is $\overline{x}\otimes \overline{y}\in (M/\fm M)\otimes_k (N/\fm N)$.

(2) Choose elements $x\in M\setminus \fm M$ and $y\in  N\setminus \fm N$ such that $\ldim_RM=\dim_R Rx$ and $\ldim_RN=\dim_R Ry$.
Using (1), we have $x\otimes y\in (M\otimes_RN) \setminus \fm (M\otimes_RN)$, and get the inequality $\ldim_R(M\otimes_RN)\leq \dim_R R(x\otimes y)$.
It remains to note that $\dim_RR(x\otimes y)\leq \min\{\dim_R Rx,\,\dim_R Ry\}$ holds.
\end{proof}

Denote by $\D(R)$ the unbounded derived category of all $R$-modules.
For an $R$-complex $X$ we denote by $\Td_RX$ the {\em restricted Tor-dimension}, which is by definition the supremum of $\sup(T\otimes_RX)$ where $T$ runs through the flat $R$-modules.
Now we can prove our main theorem.

\begin{proof}[{\bf Proof of Theorem \ref{main theorem}}]
Let $B$ be a balanced big Cohen--Macaulay $R$-module.
If $\H_0(F)=\fm\H_0(F)$, then $\ldim_R\H_0(F)=\infty$, and there is nothing to prove.
Hence, we assume $\H_0(F)\neq \fm\H_0(F)$.
It follows from Lemma \ref{ldim-tensor}(1) that $\H_0(B\otimes_R F) \cong B\otimes_R\H_0(F)\neq 0$.
In particular, we have $\inf (B\otimes_R F)=0$.
By the fact that $B$ is big Cohen--Macaulay and \cite[Theorem 2.1]{I}, we get $\dim R=\depth_R B=\depth_R(B\otimes_R F)+\sup(k\otimes_R F)$.
Thus, it suffices to show $\depth_R(B\otimes_RF)\le\ldim_R\H_0(F)$.

(a) As $\inf (B\otimes_R F)=0$, it hold that $\sup (B\otimes_R F)=0$ by assumption.
Hence $B\otimes_R F\cong\H_0(B\otimes_R F)\cong B\otimes_R\H_0(F)$ in $\D(R)$.
We have $B\otimes_R \H_0(F)\ne\fm(B\otimes_R \H_0(F))$ by Lemma \ref{ldim-tensor}(1).
Therefore $\depth_R(B\otimes_R \H_0(F))\le\ldim_R\H_0(F)$ by Proposition \ref{depth-ldim-dim} and Lemma \ref{ldim-tensor}(2).
Now the inequality $\depth_R(B\otimes_RF)\le\ldim_R\H_0(F)$ follows.

(b) Set $s=\sup(B\otimes_RF)$.
If $s=0$, then $\amp (B\otimes_RF)=0$ since $\inf (B\otimes_R F)=0$, and so we can apply (a) to deduce the assertion.
Hence, we assume $s>0$.

Let us show that $\Supp\H_i(B\otimes_RF)\subseteq \{\fm\}$ for all $i>0$.
Fix a nonmaximal prime ideal $\fp$ of $R$.
It follows from \cite[Proposition (5.3.4) and Theorems (5.3.6), (5.3.8)]{C} that
\begin{equation}\label{r1}\tag{$\diamondsuit$}
\sup (B\otimes_RF)_\fp-\sup F_\fp\le \Td_R B=\sup\{\depth R_\fq-\depth_{R_q} B_\fq\mid\fq\in \Spec R\}.
\end{equation}
If $\kappa(\fq)\Lotimes B\neq 0$, then we see from \cite[Remark 2.9]{F} and \cite[Theorems (3.2) and (3.3)]{S} that $\depth_{R_\fq}B_\fq=\dim R_\fq$.
Hence $\depth R_\fq-\depth_{R_\fq} B_\fq\leq 0$.
If $ \kappa(\fq)\Lotimes B=0$, then $\RHom_{R_\fq}(\kappa(\fq), B_\fq)=0$;  see \cite[Proposition 2.8]{F}.
In other words,  $\depth_{R_\fq}B_\fq=\infty$.
Thus, the last term of \eqref{r1} is nonpositive, which implies $\sup (B\otimes_R F)_\fp\le\sup F_\fp$.
As $\sup F_\fp\le 0$ by assumption, we obtain $\sup (B\otimes_R F)_\fp\leq 0$, which shows $\H_i(B\otimes_R F)_\fp=0$ for all $i>0$. 
Consequently, $\Supp\H_i(B\otimes_R F)\subseteq \{\fm\}$ for all $i>0$, as required.

By the proof of \cite[Chapter I, Lemma 4.6(3)]{Ha}, we get an injective resolution $I=(0\to I_s\to I_{s-1}\to\cdots)$ of the $R$-complex $B\otimes_RF$ as $\sup(B\otimes_RF)=s$, and
$$
\depth_R(B\otimes_RF)=-\sup\Gamma_\fm(I)=-\sup(0\to\Gamma_\fm(I_s)\to\Gamma_\fm(I_{s-1})\to\cdots).
$$
Since $\Supp\H_s(B\otimes_R F)\subseteq \{\fm\}$, we observe that $\Gamma_\fm(I_s)=I_s$, which implies $\H_s(\Gamma_\fm(I))=\H_s(I)=\H_s(B\otimes_RF)\ne0$.
Consequently, we get $\depth_R(B\otimes_RF)=-s<0\le\ldim_R\H_0(F)$.
\end{proof}

\begin{remark}
We may wonder whether or not in Theorem \ref{main theorem} the condition (b) implies the condition (a).
This implication does not necessarily hold even for a bounded complex of free modules of finite rank.
Indeed, suppose that $R$ is not regular.
Take a minimal system of generators $\xx=x_1,\dots,x_n$ of the maximal ideal $\fm$, and let $F=\K(\xx)$ be a Koszul complex.
Then, clearly, the condition (b) is satisfied.
As $\H_0(B\otimes F)=B/\fm B\ne0$, if (a) is satisfied as well, then $B\otimes F$ is acyclic, and it follows that
$$
\depth_RB=n-\sup(\K(\xx)\otimes_RB)=n>\dim R\ge\dim_RB,
$$
which contradicts Proposition \ref{depth-ldim-dim}.
Therefore, the condition (a) is not satisfied.

It is seen from the proof of Theorem \ref{main theorem} and \cite[Theorem 2.1]{I} that for a complex $F$ which does not satisfy (a) but satisfies (b), the inequality in Theorem \ref{main theorem} is strict.
In fact, in the above example we have
$$
\dim R < n = 0 + n = \ldim\H_0(F) + \sup(k\otimes_RF).
$$
\end{remark}

Recall that for an ideal $I$ of $R$, the {\em codimension} of $I$ is defined by $\codim I=\dim R-\dim R/I$.
For a complex $F=(0 \to F_n \xrightarrow{d_n} \cdots \xrightarrow{d_1} F_0 \to 0)$ of free $R$-modules of finite rank, the {\em codimension} of $F$ is defined by $\codim F=\inf_{1\le i\le n}(\codim\I_{r_i}(d_i)-i)$, where $r_i=\sum_{j=i}^n(-1)^{j-i}\rank F_j$ (this is called in \cite{BH} the {\em expected rank} of $d_i$).
Using our Theorem \ref{main theorem}, we recover \cite[Theorem 9.4.1]{BH}:

\begin{corollary}\label{BH thm}
Let $F=(0 \to F_n\to \cdots \to F_0 \to 0)$ be a complex of finitely generated free $R$-modules with $\codim F\geq 0$.
For each $x\in\H_0(F)\setminus \fm\H_0(F)$ one has $\codim(\ann(x)) \le n$.
\end{corollary} 

\begin{proof}
Let $B$ be a balanced big Cohen--Macaulay $R$-module.
Lemma \ref{ldim-tensor}(1) implies $\H_0(B\otimes_R F)\cong B\otimes_R H_0(F)\neq 0$.
By \cite[Lemma 9.1.8]{BH} we have $\inf(B\otimes_R F)=0$, whence $\amp(B\otimes_R F)=0$.  
Theorem \ref{main theorem} yields $\dim R\le\ldim_R H_0(F)+\sup(k\otimes_R F)\le \ldim_R H_0(F)+n \le \dim_RRx+n$, which shows $\codim(\ann(x))\le n$.
\end{proof}

\section{Little Cohen--Macaulay modules}

Let $M$ be a finitely generated $R$-module.
Recall that $M$ is called {\em Cohen--Macaulay} if $\depth_RM=\dim_RM$.
Following this, we say that $M$ is little Cohen--Macaulay if $\depth_R M=\ldim_R M$.
Also, recall that the {\em Cohen--Macaulay defect} of $M$ is defined by
$$
\cmd_R M = \dim_R M-\depth_R M.
$$
Following this, we define the {\em little Cohen--Macaulay defect} of $M$ by
$$
\lcmd_R M = \ldim_R M-\depth_R M.
$$
On the other hand, we denote by $\g_RM$ the {\em upper Gorenstein dimension} of $M$, that is, the infimum of $\pd_S(M\otimes_RR')-\pd_SR'$, where $R\to R'$ runs over the faithfully flat homomorphisms and $S\to R'$ runs over those surjective homomorphisms which satisfy $\RHom_S(R',S)\cong R'[-g]$ with $g=\pd_SR'<\infty$.

\begin{remark}\label{r2}
Let $M$ be a finitely generated $R$-module.

(1) Assume $M\ne0$.
Then $\ldim_R M\leq \dim_R M$ by Proposition \ref{depth-ldim-dim}.
Hence $\lcmd_R M\leq \cmd_R M$.

(2) There is an inequality $\g_RM \leq \pd_R M$, and the equality holds if the right-hand side is finite.
We refer the reader to \cite{V} for details.
\end{remark}

Using the little Cohen--Macaulay defect and our Theorem \ref{main theorem}, we can improve a theorem of Sharif and Yassemi \cite{SY} concerning the Cohen--Macaulay defect.

\begin{theorem}\label{scdm}
Let $M\ne0$ be a finitely generated $R$-module of finite upper Gorenstein dimension.
Then
$$
\cmd R \leq \lcmd_R M,\qquad
\dim R\leq \ldim_RM + \g_RM.
$$
\end{theorem}

\begin{proof}
The assertion follows by replacing $\cmd_R M$ and $\dim_R M$ in the proof of \cite[Theorem 2.1]{SY} with $\lcmd_R M$ and $\ldim_R M$ respectively, and using Theorem \ref{main theorem} instead of the new intersection theorem.
\end{proof}

\begin{remark}
In view of Remark \ref{r2}(1), Theorem \ref{scdm} gives a refinement of \cite[Theorem 2.1]{SY}.
\end{remark}

We obtain a couple of corollaries of Theorem \ref{scdm}.

\begin{corollary}\label{33}
Let $M\ne0$ be a finitely generated $R$-module of finite projective dimension.
Then
$$
\cmd R \leq \lcmd_R M,\qquad
\dim R\leq \ldim_RM + \pd_RM.
$$
\end{corollary}

\begin{proof}
The first inequality follows from Theorem \ref{scdm} and Remark \ref{r2}(2).
The second inequality is an immediate consequence of the first one and the Auslander--Buchsbaum formula.
\end{proof}

\begin{corollary}\label{r3}
The following are equivalent.
\begin{enumerate}[\rm(1)]
\item
The local ring $R$ is Cohen--Macaulay.
\item
There exists a Cohen--Macaulay $R$-module of finite projective dimension.
\item
There exists a Cohen--Macaulay $R$-module of finite upper Gorenstein dimension.
\item
There exists a little Cohen--Macaulay $R$-module of finite projective dimension.
\item
There exists a little Cohen--Macaulay $R$-module of finite upper Gorenstein dimension.
\end{enumerate}
\end{corollary}

\begin{proof}
The implications (2) $\Rightarrow$ (4) and (3) $\Rightarrow$ (5) are shown by Remark \ref{r2}(1), while the implications (2) $\Rightarrow$ (3) and (4) $\Rightarrow$ (5) follow from Remark \ref{r2}(2).
If $R$ is Cohen--Macaulay, then the $R$-module $R/Q$ with $Q$ a parameter ideal has finite length and finite projective dimension.
This shows (1) $\Rightarrow$ (2).
It is immediate from the first inequality in Theorem \ref{scdm} that (5) $\Rightarrow$ (1) holds.
Now the five conditions are proved to be equivalent.
\end{proof}

Let $R$ be a Cohen--Macaulay local ring.
In view of Corollary \ref{r3}, it is natural to ask if there exists a non-Cohen--Macaulay,  little Cohen--Macaulay $R$-module of finite projective dimension.
Evidently, we have to assume that $R$ has positive dimension, and then the question is actually affirmative:
Let $Q$ be a parameter ideal of $R$ and put $M=R\oplus R/Q$.
Then $\depth_R M=\ldim_R M=0$, $\dim_R M=\dim R>0$ and $\pd_R M<\infty$.
Hence it may be more meaningful to look for an indecomposable one.
However, we do not have such an example even in the case that $R$ is regular.
For example, if $R$ is a discrete valuation ring with uniformizer $x$, then every indecomposable $R$-module $M$ is isomorphic to either $R$ or $R/(x^n)$ for some $n>0$, and hence $M$ is Cohen--Macaulay. This lead us to the following modified question.

\begin{question}\label{indecomposable sCM}
Let $R$ be a regular local ring of dimension at least two.
Does there exist an indecomposable non-Cohen--Macaulay,  little Cohen--Macaulay $R$-module?
\end{question}

\begin{remark}\label{depth ldim}
Let $S\to R$ be a surjective homomorphism of (commutative noetherian) local rings.
Let $M$ be a (possibly infinitely generated) $R$-module.
Then $\depth_SM=\depth_RM$ and $\ldim_SM=\ldim_RM$.
Indeed, the first equality follows from the description of a depth by a Koszul complex.
The second one holds since $\dim_SSx=\dim_R Rx$ for any $x\in M$ and $\fn M=\fm M$, where $\fn$ is the maximal ideal of $S$.
These equalities would help us extend the above question to a homomorphic image of a regular local ring.
\end{remark}

We make observations that give some restrictions to construct a module as in Question \ref{indecomposable sCM}.
Recall that a finitely generated $R$-module $M$ is called {\em unmixed} if $\Ass_R M=\Assh_R M$, where $\Assh_R M$ stands for the set of prime ideals $\fp$ in $\Supp_RM$ such that $\dim R/\fp=\dim_R M$. 

\begin{proposition}\label{restrictions}
Let $M\ne0$ be an $R$-module.
If $M$ is cyclic or unmixed, then $\ldim_R M=\dim_R M$.
\end{proposition}

\begin{proof}
First, we consider the case where $M$ is cyclic.
Then $M\cong R/I$ for some ideal $I$ of $R$, and we have
\begin{align*}
\ldim_R R/I
&=\inf\{ \dim_R R\overline{x}\mid\overline{x}\in (R/I) \setminus (\fm/I) \}\\
&= \inf\{ \dim_R R\overline{x}\mid \text{$x$ is a unit of $R$}\}
= \dim_R R\overline{1} = \dim_R R/I.
\end{align*}

Next, we consider the case where $M$ is unmixed.
Take an element $x\in M \setminus \fm M$ satisfying $\ldim_R M=\dim_R Rx=\dim R/\ann(x)$.
Suppose $\ldim_R M < \dim_R M$.
Then for all $ \fp \in \Assh_R M$, the ideal $\ann(x) $ is not contained in $\fp$.
Using the assumption $\Ass_RM=\Assh_RM$ and prime avoidance, we find an element $y\in\ann(x)$ which is $M$-regular.
Then $yx=0$, which implies $x=0$.
This contradicts the choice of $x$.
\end{proof}

Recall that $R$ is called {\it coprimary} if $R$ has a unique associated prime. A typical example of a coprimary ring is an integral domain.
Here is a direct consequence of the above proposition.

\begin{corollary}\label{coprimary}
Suppose that $R$ is coprimary.
Let $I$ be a nonzero ideal of $R$.
Regarding $I$ as an $R$-module, one has $\ldim_R I=\dim_R I=\dim R$.
\end{corollary}

\begin{proof}
Note that there are inclusions $\emptyset \neq \Assh_R I \subseteq \Ass_R I \subseteq \Ass R$.
Since $\Ass R$ consists only of one element, one has $\Assh_R I=\Ass_R I=\Ass R$.
The assertion now follows from Proposition \ref{restrictions}(2).
\end{proof}

\section{Several examples illustrating our results}

In this section, we make observations on our results obtained in the previous sections, by presenting various examples.
In the following two examples, we consider the inequality given in Theorem \ref{main theorem}.
As we see, it is sometimes an equality, and is sometimes a strict inequality.

\begin{example}\label{example for the main theorem0}
(1) Let $F=\K(x)$ be the Koszul complex of an element $x\in\fm$. 
Then one has $\dim R=\ldim_R\H_0(F)+\sup(k\otimes_RF)$ if and only if $x$ is a subsystem of parameters of $R$.
Indeed, it is clear that $\sup(k\otimes_RF)=1$, while $\ldim_R\H_0(F)=\dim R/(x)$ by Proposition \ref{restrictions}.
Moreover, $F$ satisfies the condition (a) of Theorem \ref{main theorem} when $x$ is a subsystem of parameters.

(2) Let $R$ be a Cohen--Macaulay local ring.
Let $M\ne0$ be an $R$-module which is either cyclic or unmixed.
Assume that $M$ has finite projective dimension, and let $F$ be a minimal free resolution of $M$.
Then, by the Auslander--Buchsbaum formula, $\dim R=\ldim_R\H_0(F)+\sup(k\otimes_RF)$ if and only if $M$ is Cohen--Macaulay.
In fact, $\sup(k\otimes_RF)=\pd_RM$ and $\ldim_R\H_0(F)=\dim_RM$ by Proposition \ref{restrictions}.
Moreover, $F$ satisfies the condition (b) of Theorem \ref{main theorem} because $\Supp_R\H_i(F)=\emptyset \subseteq \{\fm\}$ for $i>0$.
\end{example}

In the next example, we treat complexes of infinitely generated flat $R$-modules.

\begin{example}\label{example for the main theorem}
(1) Let $0\to \bigoplus_{\mathbb{N}}R\to \prod_{\mathbb{N}}R\to C\to 0$ be the natural exact sequence, where $\mathbb{N}$ denotes the set of positive integers.
Then $F=(0\to \bigoplus_{\mathbb{N}}R\to \prod_{\mathbb{N}}R\to 0)$ is a flat resolution of $C$, and $\Supp\H_1(F)=\emptyset\subseteq\{\fm\}$.
There are natural isomorphisms $k\otimes_R(\bigoplus_{\mathbb{N}}R)\cong\bigoplus_{\mathbb{N}}k$ and $k\otimes_R(\prod_{\mathbb{N}}R)\cong \prod_{\mathbb{N}}k$, where the latter holds since $k$ is finitely presented.
Hence $\sup(k\otimes_R F)=0$.
This also yields $C/\fm C\neq 0$, which implies $\ldim_R C\leq \dim R$.
On the other hand, Theorem \ref{main theorem} implies $\dim R\leq \ldim_R\H_0(F)+\sup(k\otimes_RF)=\ldim_R C$.
Therefore the equality $\dim R= \ldim_R\H_0(F)+\sup(k\otimes_RF)$ holds.
 
(2) Let $x\in \fm$ and $F=\prod_{n\in\mathbb{N}}\K(x^n)$.
Note that $F$ satisfies the condition (b) of Theorem \ref{main theorem} when $x$ is a non-zerodivisor of $R$.
We have $\H_0(F)=\prod_{n\in\mathbb{N}}R/(x^n)$, and $\sup(k\otimes_R F)=1$ as $k\otimes_RF=\prod_{n\in\mathbb{N}}\K(x^n,k)$.
Also, it is seen that $\ldim_R (\prod_{n\in\mathbb{N}}R/(x^n))=\inf\{\dim_R R/(x^n)\mid n\in\mathbb{N}\}=\dim R/(x)$.
Hence, as with Example \ref{example for the main theorem0}(1), one has $\dim R=\ldim_R\H_0(F)+\sup(k\otimes_RF)$ if and only if $x$ is a subsystem of parameters of $R$.
See also Example \ref{example 2}(3).
\end{example}

Next we consider the two inequalities given in Proposition \ref{depth-ldim-dim}.

\begin{example}\label{example 1}
(1) Let $N$ be a finitely generated $R$-module with $\dim_RN>0$, and set $M=k\oplus N$.
One then has $\depth_R M=\ldim_R M=0<\dim_R M$. 

(2) Let $R=k\jump{x,y}/(x^2,xy)$ with $k$ a field.
Let $\fm=(x,y)$ be the maximal ideal of $R$.
Then we have $\depth_R \fm=\ldim_R \fm=0<1=\dim_R \fm$, and $\depth R=0<1=\ldim R=\dim R$.
In fact, note that $\fm=(x)\oplus(y)$ and $(x)\cong k$.
Hence $\ldim_R\fm=0$ by (1).
Proposition \ref{restrictions} implies $\ldim R=\dim R$.

(3) Suppose $0<\depth R<\dim R$.
Take an $R$-regular element $x \in \fm$, and set $M=R/(x)\oplus R$.
Then $\depth_R M=\depth R/(x)=\depth R-1$ and $\dim_R M=\dim R$.
Also, $\ldim_R M=\ldim_R R/(x)=\dim R/(x)=\dim R-1$, where the first equality is seen by the definition of the little dimension, while the second equality follows from Proposition \ref{restrictions}.
We conclude that $\depth_R M<\ldim_R M<\dim_R M$.
\end{example}

The third assertion of the above example gives the strict inequalities, assuming that $R$ has positive depth.
The following proposition gives the same inequalities in the case where $R$ has depth zero.

\begin{proposition}\label{depth<ldim<dim}
Assume $\depth R=0$.
Let $x$ be a minimal generator of $\fm$ such that $\fp=(x)$ is a prime ideal of $R$, $\ann(x)$ is not $\fm$-primary, $\Assh R=\{\fp\}$ and $R_\fp$ is a field.
Then, regarding $\fm$ as an $R$-module, we have the strict inequalities $\depth_R \fm<\ldim_R\fm <\dim_R \fm$.
\end{proposition}

\begin{proof}
As $\depth R=0$ and $\fm\ne0$, we have $0\ne\soc R\subseteq \fm$. 
Hence $\soc_R\fm \ne 0$, and $\depth_R\fm =0$.
As $R_\fp$ is a field, we have $\fp R_\fp=0=0R_\fp$, and $\fp=\fp R_\fp\cap R=0R_\fp\cap R$ is the $\fp$-primary component of the zero ideal $0$ of $R$.
Write $0=\fp\cap I$.
Then $I\subseteq \ann(x)$, and we have
$$
\ldim_R \fm\leq \dim_R Rx=\dim R/\ann(x)\leq \dim R/I <\dim R=\dim_R \fm,
$$
where the strict inequality follows from the assumption that $\Assh R=\{\fp\}$.

It follows from the above argument that $\dim R>0$.
If $\fp=\fm$, then $R=R_\fp$ is a field and $\fm=0$, which is a contradiction.
Hence $\fp\neq \fm$.
Suppose that there is a minimal generator $y$ of $\fm$ with $\dim_RRy=0$.
Then $\fm^ny=0$ for some $n>0$, and $\fm^ny$ is contained in the prime ideal $\fp$.
As $\fp\ne \fm$, we have $y\in \fp=(x)$ and get $y=xz$ for some $z\in R$.
Since $y\notin\fm^2$, the element $z$ must be a unit of $R$.
Thus $\fm^nx=0$ and $\ann(x)$ is $\fm$-primary, this is a contradiction.
We conclude that $\ldim_R \fm>0=\depth_R\fm$.
\end{proof}

\begin{example}\label{example 2}
Let $k$ be a field.

(1) Let $R=k\jump{x,y,z}/(x^2y,xy^2,xz)$.
Then $\ann(x)=(xy,y^2,z)$.
The zero ideal of $R$ has an irredundant primary decomposition $0=(x)\cap(y,z)\cap(x^2,y^2,z)$, which shows $\Assh R=\{(x)\}$.
Proposition \ref{depth<ldim<dim} yields $\depth_{R}\fm<\ldim_{R}\fm<\dim_{R}\fm$.
To be more precise, $\depth_{R}\fm=0$, $\ldim_{R}\fm=1$ and $\dim_{R}\fm=2$.

(2) There is also an example of an equidimensional local ring.
Let $R=k\jump{x,y,z}/(x^2,xy^2,xyz)$.
Then $\ann(x)=(x,y^2,yz)$.
We have $0=(x)\cap(y)\cap(y^2,z)$, and $\Min R=\Assh R=\{(x)\}$.
Proposition \ref{depth<ldim<dim} implies $\depth_{R}\fm<\ldim_{R}\fm<\dim_{R}\fm$.
In fact, we have $\depth_{R}\fm=0$, $\ldim_{R}\fm=1$ and $\dim_{R}\fm=2$.

(3) Let us present an example of an infinitely generated module.
Take an element $x\in\fm$, and set $M=\prod_{n\in\mathbb{N}}R/(x^n)$.
Then there is an inclusion $\widehat R=\varprojlim_{n\in \mathbb{N}} R/(x^n)\hookrightarrow M$, where $\widehat R$ denotes the $(x)$-adic completion of $R$.
The inclusions $R\hookrightarrow \widehat R\hookrightarrow M$ yield $\dim R=\dim_R\widehat R=\dim_RM$.
Now suppose that $x$ is a non-zerodivisor. Then, there is an isomorphism $\RHom_R(k, M)\cong \prod_{n\in\mathbb{N}}\RHom_R(k, R/(x^n))$, from which we obtain $\depth_R M=\depth R -1$.
We see from Example \ref{example for the main theorem}(2) that $\ldim_RM=\dim R/(x)=\dim R-1$.
Thus, under the assumption that $R$ is not Cohen--Macaulay, we have $\depth_RM< \ldim_RM< \dim_RM$. 
\end{example}

Let $M$ be a nonzero finitely generated $R$-module of finite projective dimension.
Combining the first inequality in Corollary \ref{33} with Remark \ref{r2}(1), we have $\cmd R \le \lcmd_R M \le \cmd_R M$.
We give examples where either/both of these inequalities become strict. 

\begin{example}\label{example cmd}
(1) Let $R$ and $M$ be as in Example \ref{example 1}(3).
Then $M$ has projective dimension one (hence finite), and it holds that $\cmd R=\lcmd_R M<\cmd_R M=\cmd R+1$.

(2) Suppose that $R$ is regular and $\dim R\geq 2$.
Then Corollary \ref{coprimary} implies $\ldim_R \fm=\dim_R\fm=\dim R$.
As $\depth_R \fm =1$, we have $\cmd R=0 <\lcmd_R\fm=\cmd_R\fm=\dim R-1$.
(Here, the regularity of $R$ is needed just to have that $\pd_R\fm$ is finite.
More precisely, we have $\cmd R <\lcmd_R\fm=\cmd_R\fm=\dim R-1$ for any coprimary local ring $R$ with $\depth R\geq 2$.)

(3) Let $S=k\jump{x,y,z}$ with $k$ a field, and set $R=S/(x^2y,xy^2,xz)$.
Let $\fm=(x,y,z)R$ be the maximal ideal of $R$. 
By Example \ref{example 2}(1), we have $\depth_R \fm=0$, $\ldim_R \fm=1$ and $\dim_R \fm=2$.
Now we regard $\fm$ as an $S$-module.
Then $\depth_S \fm=0$, $\ldim_S \fm=1$ and $\dim_S \fm=2$ by Remark \ref{depth ldim}.
It holds that $\cmd_S S=0$, $\lcmd_S \fm=1$ and $\cmd_S \fm=2$.
Thus $\cmd_S S<\lcmd_S \fm <\cmd_S \fm$ and $\pd_S \fm=3 <\infty$.
On the other hand, we have $\cmd_R R=2>1=\lcmd_R\fm$ and $\pd_R \fm=\infty$.
This shows that the assumption that $M$ has finite projective dimension is necessary for the first inequality in Corollary \ref{33} to hold true.
\end{example}

\begin{ac}
The basic ideas of this paper were conceived by Hans-Bj\o rn Foxby, more than five years after Siamak Yassemi finished his formal apprenticeship with Foxby \cite{FY}.
The authors would like to dedicate this paper to their friend Hans-Bj\o rn Foxby for sharing his useful ideas with them. Part of this work was completed while Siamak Yassemi was visiting the Max Planck Institute in Bonn-Germany.
He wishes to express his gratitude to the Institute for its warm hospitality and for providing a stimulating research environment. In addition, Tsutomu Nakamura acknowledges support from the Program Ricerca di Base 2015 of the University of Verona. Finally, the authors thank the referee for reading the paper carefully and giving them valuable comments.
\end{ac}



\end{document}